\documentclass{amsart}

\usepackage{latexsym,enumerate}
\usepackage{amsmath,amsthm,amsopn,amstext,amscd,amsfonts,amssymb,fontenc}
\usepackage[ansinew]{inputenc}
\usepackage{verbatim}
\usepackage{graphicx}
\usepackage{txfonts,pxfonts,wasysym}
\usepackage{pstricks,pst-node}

\theoremstyle{theorem}
\newtheorem{theorem}{Theorem}

\newtheorem{conjecture}{Conjecture}
\newtheorem{lemma}{Lemma}
\newtheorem{remark}{Remark}

\theoremstyle{definition}

\renewcommand{\a}{\alpha}
\renewcommand{\b}{\beta}
\newcommand{\e}{\varepsilon}
\renewcommand{\d}{\delta}
\newcommand{\g}{\gamma}

\newcommand{\p}{\partial}

\begin{document}

\title[Solving a Conjecture]{Solving a conjecture about tessellation graphs of $\mathbb R^2$}

\author[Walter Carballosa]{Walter Carballosa}
\address{Consejo Nacional de Ciencia y Tecnolog\'ia (CONACYT) \& Universidad Aut\'onoma de Zacatecas,
Paseo la Bufa, int. Calzada Solidaridad, 98060 Zacatecas, ZAC, M\'exico}
\email{wcarballosato@conacyt.mx}

%
%\author[Jos\'e M. Rodr{\'\i}guez]{Jos\'e M. Rodr{\'\i}guez}
%\address{Departamento de Matem\'aticas, Universidad Carlos III de Madrid,
%Avenida de la Universidad 30, 28911 Legan\'es, Madrid, Spain}
%\email{jomaro@math.uc3m.es}
%%\thanks{$^{(2)}$ Supported in part by a grant from CONACYT \ (CONACYT-UAG I0110/62/10), M\'exico.}

\date{\today}

\maketitle{}

\begin{abstract}
In the paper \emph{Planarity and Hyperbolicity in Graphs}, the authors present the following conjecture: every tessellation of the Euclidean plane with convex tiles induces a non-hyperbolic graph.
It is natural to think that this statement holds since the Euclidean plane is non-hyperbolic.
Furthermore, there are several results supporting this conjecture.
However, this work shows that the conjecture is false.
%Furthermore, we prove that every Delaunay triangulation of the Euclidean plane is non-hyperbolic.
\end{abstract}



\
\section{Introduction.}%{Introduction}
Hyperbolic spaces play an important role in geometric
group theory and in the geometry of negatively curved
spaces (see \cite{ABCD, GH, G1}).
The concept of Gromov hyperbolicity grasps the essence of negatively curved
spaces like the classical hyperbolic space, Riemannian manifolds of
negative sectional curvature bounded away from $0$, and of discrete spaces like trees
and the Cayley graphs of many finitely generated groups. It is remarkable
that a simple concept leads to such a rich
general theory (see \cite{ABCD, GH, G1}).

The first works on Gromov hyperbolic spaces deal with
finitely generated groups (see \cite{G1}). %G2
Initially, Gromov spaces were applied to the study of automatic groups in the science of computation
(see, \emph{e.g.}, \cite{O}); indeed, hyperbolic groups are strongly geodesically automatic, \emph{i.e.}, there is an automatic structure on the group \cite{Cha}.

The concept of hyperbolicity appears also in discrete mathematics, algorithms
and networking. For example, it has been shown empirically
in \cite{ShTa} that the internet topology embeds with better accuracy
into a hyperbolic space than into an Euclidean space
of comparable dimension; the same holds for many complex networks, see \cite{KPKVB}.
A few algorithmic problems in
hyperbolic spaces and hyperbolic graphs have been considered
in recent papers (see \cite{ChEs, Epp, GaLy, Kra}).
Another important
application of these spaces is the study of the spread of viruses through on the
internet (see \cite{K21,K22}).
Furthermore, hyperbolic spaces are useful in secure transmission of information on the
network (see \cite{K27,K21,K22,NS}).

The study of Gromov hyperbolic graphs is a subject of increasing interest; see, \emph{e.g.}, \cite{BRSV2,CGPR1,CGPR2,CGPR3,CPoRS,Ham,RST,S2} and the references therein.

There are several definitions of Gromov hyperbolicity.
These different definitions are equivalent in the sense that if $X$ is $\d$-hyperbolic with respect to the definition $A$, then it is $\d'$-hyperbolic with respect to the definition $B$ for some $\d'$ (see, e.g., \cite{BHB,GH}).

Given a metric space $X$, we define the \emph{Gromov product} of $x,y\in X$ with base point $w\in X$ by
\begin{equation}\label{eq:productGromov}
(x|y)_w:= \frac12\, \big(d(x,w)+d(y,w) -d(x,y) \big) .
\end{equation}
We say that the \emph{Gromov product is $\d$ hyperbolic} if there is a constant $\d\geq 0$ such that
\begin{equation}\label{eq:GProd}
(x|y)_w
\ge \min \big\{ (x|z)_w, (z|y)_w \big\}- \d
\end{equation}
for every $x,y,z\in X$ and some $w\in X$.
It is well known that this definition is independent of base point, see \cite[Proposition 2.2]{ABCD} and \cite[Lemma 1.1A]{G1}.
In fact, if $X$ is a metric space, $w,w'\in X$ and the Gromov product based at $w$ is $\d$-hyperbolic, then the Gromov product based at $w'$ is $2\d$-hyperbolic.
We say that $X$ is {\it hyperbolic} if there exists a constant $\d\ge0$ such that its Gromov product is $\d$-hyperbolic for any base point, see, \emph{e.g.}, \cite{GH}.

We say that the curve $\g$ in a metric space $X$ is a
\emph{geodesic} if we have $L(\g|_{[t,s]})=d(\g(t),\g(s))=|t-s|$ for every $s,t\in [a,b]$
(then $\g$ is equipped with an arc-length parametrization).
The metric space $X$ is said \emph{geodesic} if for every couple of points in
$X$ there exists a geodesic joining them; we denote by $[xy]$
any geodesic joining $x$ and $y$; this notation is ambiguous, since in general we do not have uniqueness of
geodesics, but it is very convenient.
Consequently, any geodesic metric space is connected.
If the metric space $X$ is a graph, then the edge joining the vertices $u$ and $v$ will be denoted by $[u,v]$.
%It is clear that every geodesic metric space is path-connected.

In order to consider a graph $G$ as a geodesic metric space, we must identify
any edge $[u,v]\in E(G)$ with the real interval $[0,l]$, if $l:=L([u,v])$; therefore, any point in the interior of any edge is a point of $G$.
A connected graph $G$ is naturally equipped with a distance
defined on its points, induced by taking shortest paths in $G$.
Then, we see $G$ as a geodesic metric graph.

Throughout the paper we just deal with graphs which are connected and locally finite
(i.e., each ball of finite radius contains just a finite number of edges); we also allow edges of arbitrary length.
These conditions guarantee that the graph is a geodesic metric space.
In particular, we pay attention to the planar graphs: the graphs which are the ``boundary" (the $1$-skeleton)
of a tessellation of the Euclidean plane.

By a \emph{tessellation graph} we mean the $1$-skeleton (\emph{i.e.}, the set of $1$-cells) of
a CW $2$-complex contained in a complete Riemannian surface $\mathcal S$
(with or without boundary)
such that every point in $\mathcal S$
is contained in some face ($2$-cell) of the complex.
We assume that every closed cell is embedded in the CW-complex, \emph{i.e.},
a face of the tessellation should not be glued to
itself along a pair of edges.
The edges of a tessellation graph are just rectifiable paths in $\mathcal S$
and have the length induced by the metric in $\mathcal S$
(they can be either geodesics or not in $\mathcal S$).
Note that this class of graphs contains as particular cases many planar graphs.

In \cite{CPoRS} the authors present the following conjecture ``every tessellation of the Euclidean plane with convex tiles induces a non-hyperbolic graph".
It is natural to think that this statement holds since the Euclidean plane is non-hyperbolic.
Furthermore, there are several theorems in \cite{CPoRS} supporting this conjecture.
This work shows that the conjecture is false, see Section \ref{sect4}.
%Finally, in Section \ref{sectDT} we prove that every Delaunay triangulation of the Euclidean plane is non-hyperbolic, see Theorem \ref{t:Delaunay}.

\
\section{The conjecture and previous results}
The conjecture was presented in \cite{CPoRS} as follows.

\begin{conjecture}\label{conject}
Every tessellation graph of $\mathbb R^2\simeq\mathbb C$ with convex tiles is non-hyperbolic.
\end{conjecture}

Note that many tessellation graphs of the Euclidean plane $\mathbb R^2$ are non-hyperbolic, and many tessellation graphs of hyperbolic space as the Poncar\'e disk $\mathbb D$ are hyperbolic.
One can think that the tessellation graphs of the Euclidean plane $\mathbb R^2$ are always non-hyperbolic, since the plane is non-hyperbolic;
however, in \cite{PRSV} the authors show a hyperbolic tessellation graph of $\mathbb R^2$.
The main aim on it is enlarge the edges of a non-hyperbolic graph until to obtain a hyperbolic graph (similar to a tessellation graph of $\mathbb D$). However, this example is a tessellation graph of $\mathbb R^2$ with not convex tiles.
Furthermore, it is easy to obtain a tessellation graph of $\mathbb D$ which is non-hyperbolic, taking huge tiles on $\mathbb D$.

The following results support the Conjecture \ref{conject}.

\begin{theorem}\cite[Theorem 2.7]{CPoRS}
\label{t:convexarea}
Given any tessellation graph $G$ of $\mathbb R^2$ with convex tiles $\{F_n\}$. If $\inf_n A(F_n)>0$, then $G$ is not hyperbolic.
%Suppose that a graph $G$ is the $1$-skeleton of a tessellation of $\mathbb R^2$ with convex tiles $\{F_n\}$. If $\inf_n A(F_n)>0$, then $G$ is not hyperbolic.
\end{theorem}

\begin{theorem}\cite[Theorem 2.9]{CPoRS}
\label{t:Planar2}
%Suppose that a graph $G$ is the $1$-skeleton of a tessellation of $\mathbb R^2$ with convex tiles $\{F_n\}$.
Given any tessellation graph $G$ of $\mathbb R^2$ with convex tiles $\{F_n\}$.
Let us assume that there exist balls $B_{n}\subset F_n$
with radius $r_{n}$ such that $L(\p F_{n}) \le c_1 r_n$ for some positive constant $c_1$ and for every $n$.
Then $G$ is not hyperbolic.
\end{theorem}

In \cite{CGPR2} the authors trying to solve the Conjecture \ref{conject} show that under appropriate assumptions adding or removing an infinite amount of edges to a given planar graph preserves its non-hyperbolicity.
Also, a partial answer to Conjecture \ref{conject} was given. It is shown that in order to prove this conjecture it suffices to consider tessellations graphs of $\mathbb R^2$ such that every tile is a triangle. %, see Theorem \ref{t:ConjTriang}.

\begin{theorem}\cite[Theorem 5.1]{CGPR2}
\label{t:ConjTriang}
All tessellation graphs of $\mathbb R^2$ whose tiles are convex polygons are non-hyperbolic if and only if all tessellation graphs of $\mathbb R^2$ whose tiles are triangles are non-hyperbolic.
\end{theorem}

The following definitions and results will be useful in Section \ref{sect4}.

Let $(X,d_X)$ and $(Y,d_Y)$  be two (geodesic) metric spaces. A map $f: X\longrightarrow Y$ is said to be
an $(\alpha, \beta)$-\emph{quasi-isometric embedding}, with constants $\alpha\geq 1,\
\beta\geq 0$ if for every $x, y\in X$:
$$
\alpha^{-1}d_X(x,y)-\beta\leq d_Y(f(x),f(y))\leq \alpha d_X(x,y)+\beta.
$$
The function $f$ is $\varepsilon$-\emph{full} if
for each $y \in Y$ there exists $x\in X$ with $d_Y(f(x),y)\leq \varepsilon$.

A map $f: X\longrightarrow Y$ is said to be
a \emph{quasi-isometry} if there exist constants $\alpha\geq 1,\
\beta,\varepsilon \geq 0$ such that $f$ is an $\varepsilon$-full
$(\alpha, \beta)$-quasi-isometric embedding.

\smallskip

A fundamental property of hyperbolic spaces is the following:

\begin{theorem}[Invariance of hyperbolicity]\label{invarianza}
Let $f:X\longrightarrow Y$ be an $(\alpha,\beta)$-quasi-isometric embedding between the geodesic metric spaces $X$ and $Y$.
If $Y$ is hyperbolic, then $X$ is hyperbolic.
%Furthermore, if $Y$ is $\d$-hyperbolic, then $X$ is $\d'$-hyperbolic, where $\d'$ is a constant which just depends on $\alpha,\beta,\d$.

Furthermore, if $f$ is $\e$-full for some $\e\ge0$ (a quasi-isometry), then $X$ is
hyperbolic if and only if $Y$ is hyperbolic.
%Furthermore, if $X$ is $\d$-hyperbolic, then $Y$ is $\d'$-hyperbolic, where $\d'$ is a constant which just depends on $\alpha,\beta,\d,\e$.
\end{theorem}

We say that a vertex $v$ of a graph $G$ is a \emph{cut-vertex} if
$G \setminus \{v\}$ is not connected.
Given a graph $G$, a family of subgraphs $\{G_n\}_{n\in \Lambda}$ of
$G$ is a \emph{T-decomposition} of $G$ if $\cup_n G_n=G$ and $G_n\cap G_m$ is either a cut-vertex or the empty set for each $n\neq m$.
We will need the following result, which allows to obtain global information about the
hyperbolicity of a graph from local information.

\begin{theorem}\cite[Theorem 5]{BRSV2}
\label{t:treedec}
Let $G$ be any graph and let $\{G_n\}_n$ be any
T-decomposition of $G$. Then $\d(G)=\sup_n \d(G_n)$.
\end{theorem}

\
\section{Solving the conjecture}\label{sect4}

In \cite{CGPR1,CGPR3} the authors obtain interesting results characterizing the hyperbolicity of many periodic graphs.
Hence, we deal with periodic graphs.

A \emph{geodesic line} is a geodesic with domain $\mathbb R$.
We say that a graph $G$ is \emph{periodic} if there exist a geodesic line $\g_0$ and an isometry $T$ of $G$ with the following properties:
\begin{itemize}
\item[(1)] {$T\g_0 \cap \g_0 = \emptyset$,}
\item[(2)] {$G \setminus\g_0$ has two connected components,}
\item[(3)] {$G \setminus \{\g_0 \cup T\g_0\}$ has at least three connected components, two of them, $G_1$ and $G_2$, satisfy $\p G_1 \subset \g_0$ and $\p G_2 \subset T \g_0$, and the subgraph $G':= G \setminus \{G_1 \cup G_2\}$ is connected and $\cup_{n\in\mathbb{Z}}T^n(G')=G$.}
\end{itemize}

Such subgraph $G'$ is called a \emph{period graph} of $G$.
In what follows, $G$ will denote a periodic graph and $G'$ a period graph of $G$.

The following result will be useful.

\begin{theorem}\cite[Theorem 1.1]{CGPR3}
\label{t:periodic}
Let $G$ be a periodic graph with $\inf_{z\in\g_0} d_G(z, T z) > 0$. Then $G$ is hyperbolic if only if $G'$ is
hyperbolic and $\lim_{|z|\to\infty,z\in\g_0} d_G(z, T z) = \infty$.
\end{theorem}

%In order to provide a tessellation graph of $\mathbb R^2$ which is hyperbolic a graph $G'$ is shown in Figure \ref{fig:Conj}.

\begin{figure}[h]
  \centering
  \scalebox{1.9}
  {\begin{pspicture}(1.7,-.2)(8.3,1.7)
  \psline[linewidth=0.02cm,linecolor=black,linestyle=dashed](1.7,1)(8.3,1)
  \psline[linewidth=0.02cm,linecolor=black](2,1)(8,1)
  \psline[linewidth=0.02cm,linecolor=black,linestyle=dashed](1.7,0)(8.3,0)
  \psline[linewidth=0.02cm,linecolor=black](2,0)(8,0)
%  \psline[linewidth=0.02cm,linecolor=black](2,0)(2,1)
  \psline[linewidth=0.02cm,linecolor=black,linestyle=dashed](1.7,0.35)(2,0.5)(1.7,0.65)
  \psline[linewidth=0.02cm,linecolor=black,linestyle=dashed](2,0.5)(1.7,0.5)
  \psline[linewidth=0.02cm,linecolor=gray](2,0.5)(3,0)(4,0.25)(3,0.5)(2,0.5)
  \psline[linewidth=0.02cm,linecolor=gray](2,0.5)(3,1)(4,0.75)(3,0.5)
  \psline[linewidth=0.02cm,linecolor=black](4,0.25)(5,0)(6,0.125)(5,0.25)(4,0.25)
  \psline[linewidth=0.02cm,linecolor=black](4,0.25)(5,0.5)(6,0.375)(5,0.25)
  \psline[linewidth=0.02cm,linecolor=black](4,0.75)(5,0.5)(6,0.625)(5,0.75)(4,0.75)
  \psline[linewidth=0.02cm,linecolor=black](4,0.75)(5,1)(6,0.875)(5,0.75)
  \psline[linewidth=0.02cm,linecolor=black](6,0.125)(7,0)(8,0.0625)(7,0.125)(6,0.125)
  \psline[linewidth=0.02cm,linecolor=black](6,0.125)(7,0.25)(8,0.1875)(7,0.125)
  \psline[linewidth=0.02cm,linecolor=black](6,0.375)(7,0.25)(8,0.3125)(7,0.375)(6,0.375)
  \psline[linewidth=0.02cm,linecolor=black](6,0.375)(7,0.5)(8,0.4375)(7,0.375)
  \psline[linewidth=0.02cm,linecolor=black](6,0.625)(7,0.5)(8,0.5625)(7,0.625)(6,0.625)
  \psline[linewidth=0.02cm,linecolor=black](6,0.625)(7,0.75)(8,0.6875)(7,0.625)
  \psline[linewidth=0.02cm,linecolor=black](6,0.875)(7,0.75)(8,0.8125)(7,0.875)(6,0.875)
  \psline[linewidth=0.02cm,linecolor=black](6,0.875)(7,1)(8,0.9375)(7,0.875)
  \psline[linewidth=0.01cm,linecolor=black,linestyle=dashed](8,0.0625)(8.3,0.0625)
  \psline[linewidth=0.01cm,linecolor=black,linestyle=dashed](8,0.1875)(8.3,0.1875)
  \psline[linewidth=0.01cm,linecolor=black,linestyle=dashed](8,0.3125)(8.3,0.3125)
  \psline[linewidth=0.01cm,linecolor=black,linestyle=dashed](8,0.4375)(8.3,0.4375)
  \psline[linewidth=0.01cm,linecolor=black,linestyle=dashed](8,0.5625)(8.3,0.5625)
  \psline[linewidth=0.01cm,linecolor=black,linestyle=dashed](8,0.6875)(8.3,0.6875)
  \psline[linewidth=0.01cm,linecolor=black,linestyle=dashed](8,0.8125)(8.3,0.8125)
  \psline[linewidth=0.01cm,linecolor=black,linestyle=dashed](8,0.9375)(8.3,0.9375)
  \uput[90](2,1.1){$_0$}
  \uput[90](3,1.1){$_1$}
  \uput[90](4,1.1){$_2$}
  \uput[90](5,1.1){$_3$}
  \uput[90](6,1.1){$_4$}
  \uput[90](7,1.1){$_5$}
  \uput[90](8,1.1){$_6$}
  \end{pspicture}}
  \caption{Period graph $G'$ of a tessellation graph of $\mathbb R^2$ with convex tiles.}
  \label{fig:Conj}
\end{figure}
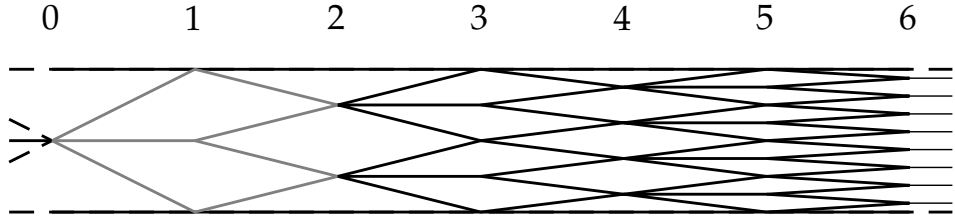

Let $G'$ be the graph showed in Figure \ref{fig:Conj}, \emph{i.e.}, $G'$ is a planar graph contained in $\mathbb R\times [0,1] \subset \mathbb R^2$ verifying the following properties
\begin{itemize}
  \item[(1)] {$\mathbb R\times\{0,1\} \subset G'$,}
  \item[(2)] {$G'$ is symmetric\footnote{Figure \ref{fig:Conj} shows a half ($\mathbb{R}e\,z >0$) of the graph $G'$.} with respect to the vertical axis,}
  \item[(3)] {in $G\cap\{z\in\mathbb C : 2n-2\le \mathbb Re \le 2n\}$ ($n\ge1$) there are $2^{n-1}$ graphs isomorphic to the subgraph in gray on Figure \ref{fig:Conj}.}
\end{itemize}
We say that a vertex $v$ of $G'$ is at level $i$ if $v$ belongs to $\{i\}\times[0,1]$; besides, $G'$ just has $2^n$ vertices at level $2n$. Denote these vertices by $v_{2n,i}$ for $1\le i\le 2^n$ with
\[
\mathbb Im \, v_{2n,i} < \mathbb Im \, v_{2n,i+1},\quad \text{for every } 1\le i\le 2^n.
\]
Consider now the tessellation graph $W$ of $\mathbb R^2$ with period graph $G'$.

In order to prove that $W$ is hyperbolic denote by
\[
  G'':=\overline{G'\setminus \mathbb R\times\{0,1\}} \bigcap \big\{z\in\mathbb C : \mathbb{R}e\,z \ge0\big\},
\]
and $G''_1$ a graph isomorphic to $G''$ with every edge of length $1$.
Note that $G''_1$ is not necessary embedded in $\mathbb R^2$; besides, every vertex of $G''_1$ at level $i$ is at distance $i$ to the vertex $v_{0,1}$.

\begin{lemma}\label{prop:DistLevels}
Let $k,n,m$ be positive integers with $1\le k\le 2^n$.
Then, $d_{G''_1}(v_{2n,k}\;,v_{2n+2m,k_m})=2m$ if $\max\{2^m k - 2^{m+1}+2, 1\} \le k_m\le \min\{ 2^m k + 2^m - 1,2^{n+m}\}$; otherwise, the distance is greater than $2m+1$.
\end{lemma}

\begin{proof}
Note that distance between two vertices at different levels is at least the difference of their levels. Clearly $d_{G''_1}(v_{2n,k}\;,v_{2n+2m,k_m})$ is an even number.
Hence, it suffices to prove that $d_{G''_1}(v_{2n,k}\;, v_{2n+2m,2^m k + 2^m -1})= 2m$ but $d_{G''_1}(v_{2n,k}\;, v_{2n+2m,2^m k + 2^m})> 2m$ if $2^m k + 2^m -1\le 2^{n+m}$ and $d_{G''_1}(v_{2n,k}\;, v_{2n+2m,2^m k - 2^{m+1} +2})= 2m$ but $d_{G''_1}(v_{2n,k}\;, v_{2n+2m,2^m k - 2^{m+1}+1})> 2m$ if $2^m k - 2^{m+1}+1 \ge 1$.
Note that $d_{G''_1}(v_{2n,k}\;,v_{2n+2,k_1})=2$ if $\max\{2k-2,1\} \le k_1\le \min\{2k+1,2^{n+1}\}$; otherwise, the distance is greater than $3$.
Then, since $1+2+\ldots+2^{m-1}=2^m-1$, we have $d_{G''_1}(v_{2n,k}\;, v_{2n+2m,2^m k + 2^m -1})= 2m$ if $k\le 2^n-1$ (\emph{i.e.}, $2^m k + 2^m -1\le 2^{n+m}$); furthermore, it follows easily that $d_{G''_1}(v_{2n,k}\;, v_{2n+2m,2^m k + 2^m})> 2m$ if $k \le 2^{n}-1$.
Similarly, since $2^m+\ldots+2=2^{m+1}-2$, we obtain $\ d_{G''_1}(v_{2n,k}\;, v_{2n+2m,2^m k - 2^{m+1} +2})= 2m$ and $d_{G''_1}(v_{2n,k}\;, v_{2n+2m,2^m k - 2^{m+1}+1})> 2m$ if $k \ge 2$.
Thus, we have $d_{G''_1}(v_{2n,k}\;,v_{2n+2m,k_m})=2m$ for $\max\{2^m k - 2^{m+1}+2, 1\} \le k_m\le \min\{ 2^m k + 2^m - 1,2^{n+m}\}$; otherwise, the distance is greater than $2m+1$.
\end{proof}

We will use the following facts.

\begin{remark}\label{r:DistLevelsUp}
For $1\le k\le 2^{n}$ the vertex $v_{2n,k}$ is at distance $2m$ to at most $2^{m+1}+2^m-2$ vertices at level $2n+2m$. Moreover, since $m\ge1$ we have $2^{m+1}\le 2^{m+1}+2^m-2< 2^{m+2}$.
\end{remark}

\begin{remark}\label{r:DistLevelsDown}
For $1\le l\le 2^{n+m}$ the vertex $v_{2(n+m),l}$ is at distance $2m$ to at most three vertices at level $2n$, and furthermore these vertices at level $2n$ have consecutive indices.
\end{remark}

\begin{lemma}\label{l:distEx}
Let $v_{2n,i},v_{2n,j}$ be two different vertices of $G''_1$ at level $2n$. Then, there is a fixed constant $\a$ such that
\begin{equation}\label{eq:distEx}
      \left|d_{G''_1}(v_{2n,i}\;,v_{2n,j}) - 4 \log_2 |i-j|\right| \le \a.
\end{equation}
\end{lemma}

\begin{proof}
Without loss of generality we can assume that $i< j$.
Note that if $j-i\le 3$ then $d_{G''_1}(v_{2n,i}\;,v_{2n,j})\le6$. %, otherwise the distance is greater than $7$.
Hence, \eqref{eq:distEx} holds in this case for every $\a\ge2$. Assume that $j-i\ge4$.

Clearly, $d_{G''_1}(v_{2n,i}\;,v_{2n,j})$ is even.
Note that every path joining $v_{2n,i}$ and $v_{2n,j}$ without vertices at levels lower than $2n-1$ has length at least $2(j-i)$.
%Moreover, the exponential growth of the number of vertices per level makes geodesic drop.
Let $r$ be the positive integer such that $2^r\le j-i< 2^{r+1}$, \emph{i.e.}, $r=\left\lfloor \log_2 (j-i) \right\rfloor$.
On the one hand, since $j-i < 2^{r+1}$ by Remark \ref{r:DistLevelsUp} there is an integer $k$ such that
$$2^{r+1}k-2^{r+2}+2\le i < j\le 2^{r+1}k+2^{r+1}-1.$$
Thus we have $d_{G''_1}(v_{2n,i}\;,v_{2(n-r-1),k})=d_{G''_1}(v_{2(n-r-1),k}\;,v_{2n,j})=2r+2$ and so $d_{G''_1}(v_{2n,i}\;,v_{2n,j})\le 4r+4$.
On the other hand, since $j-i\ge 2^{r}$ by Remark \ref{r:DistLevelsUp} there is not a vertex $v_{2(n-r+2),k}$ for some $1\le k\le 2^{n-r+2}$ with \begin{equation}\label{eq:aux}
  d_{G''_1}(v_{2n,i},v_{2(n-r+2),k})=d_{G''_1}(v_{2(n-r+2),k},v_{2n,j})=2r-4.
\end{equation}
In order to prove that $d_{G''_1}(v_{2n,i}\;,v_{2n,j})> 4(r-2)$ we proceed by contradiction.
Assume that $d_{G''_1}(v_{2n,i}\;,v_{2n,j})\le 4(r-2)$.
Hence, there is a geodesic $\g:=[v_{2n,i} v_{2n,j}]$ joining $v_{2n,i}$ and $v_{2n,j}$ which is contained between levels $2n-2r+5$ and $2n$.
Denote by $2m$ the lower even level with vertices in $\g$.
Let $v_{2m,i'},v_{2m,j'}$ be the vertices in $\g$ at level $2m$ with the maximum distance between them.
Thus, $\g:=[v_{2n,i} v_{2m,i'}]\cup[v_{2m,i'} v_{2m,j'}]\cup[v_{2m,j'} v_{2n,j}]$.
We have $d_{G''_1}(v_{2n,i},v_{2m,i'}),d_{G''_1}(v_{2n,j},v_{2m,j'})\ge 2n-2m$, and so,
\[
\begin{aligned}
d_{G''_1}(v_{2m,i'},v_{2m,j'})&= L(\g)- d_{G''_1}(v_{2n,i},v_{2m,i'}) - d_{G''_1}(v_{2n,j},v_{2m,j'}) \\
& \le 4(r-2) - 4(n-m) \\
&= 2\big(2m-2(n-r+2)\big).
\end{aligned}
\]
But, since the subgeodesic $[v_{2m,i'} v_{2m,j'}]\subset\g$ joining $v_{2m,i'}$ and $v_{2m,j'}$ contains their vertices at level $2m-1$ or upper levels we have that $|i'-j'|\le 2m-2(n-r+2)$.
Thus, by Lemma \ref{prop:DistLevels} there is a vertex $v_{2(n-r+2),k}$ at level $2(n-r+2)$ verifying $d_{G''_1}(v_{2m,i'},v_{2(n-r+2),k})=d_{G''_1}(v_{2(n-r+2),k},v_{2m,j'})=2m-2(n-r+2)$ and so, \eqref{eq:aux} holds. This is the contradiction we were looking for, and then $d_{G''_1}(v_{2n,i}\;,v_{2n,j})> 4(r-2)$.
Therefore, we have $4 r - 8 < d_{G''_1}(v_{2n,i}\;,v_{2n,j}) \le 4 r + 4$.
These inequalities finish the proof since $r=\left\lfloor \log_2 (j-i) \right\rfloor$.
\end{proof}

\begin{lemma}\label{l:ProdLevels}
For any positive integers $n,m$, consider two vertices $v_{2n,i}$, $v_{2n,j}$ of $G''_1$ at level $2n$ and two vertices $v_{2(n+m),i_0}$, $v_{2(n+m),j_0}$ at level $2(n+m)$ with $d_{G''_1}(v_{2n,i},v_{2(n+m),i_0})=2m$ and $d_{G''_1}(v_{2n,j},v_{2(n+m),j_0})=2m$. Then, there is a fixed constant $\b$ such that
\[
\left|(v_{2n,i}|v_{2n,j})_{v_{0,1}} - (v_{2(n+m),i_0}|v_{2(n+m),j_0})_{v_{0,1}}\right| \le \b.
\]
\end{lemma}

\begin{proof}
Recall that
\[
d_{G_1''}(v_{0,1},v_{2r,s})=2r \quad \text{for every} \quad r\ge0, \, 1\le s\le 2^r.
\]
By Lemma \ref{l:distEx} we have
\[
2n - 2\log_2 |i-j| -\frac{\a}2\le (v_{2n,i}|v_{2n,j})_{v_{0,1}} \le 2n - 2\log_2 |i-j| + \frac{\a}2
\]
and
\[
2(n+m) - 2\log_2 |i_0-j_0| -\frac{\a}2 \le (v_{2(n+m),i_0}|v_{2(n+m),j_0})_{v_{0,1}} \le 2(n+m) - 2\log_2 |i_0-j_0| +\frac{\a}2.
\]
Hence, by Lemma \ref{prop:DistLevels} we have
\[
2^m i -2^{m+1}+2\le i_0\le 2^m i +2^m-1
\]
and
\[
2^m j -2^{m+1}+2\le j_0\le 2^m j + 2^m-1.
\]
Therefore, we have
\[
2^m ( |i-j| -3) < |i_0-j_0| < 2^m ( |i-j| +3)
\]
and the result follows.
\end{proof}

\begin{lemma}\label{l:Ghyp}
 The graph $G''_1$ is hyperbolic.
\end{lemma}

\begin{proof}
In order to prove that the Gromov product based at $w$ is hyperbolic it suffices to obtain
\begin{equation}\label{eq:hyp}
\sup_{x,y,z} \left\{ \min \big\{ (x|z)_w - (x|y)_w, (z|y)_w - (x|y)_w \big\}\right\} < \infty.
\end{equation}
We will prove that the Gromov product based at $w:=v_{0,1}$ is hyperbolic. Note that since every point in $G''_1$ is at distance at most $1$ of the set $V_E$ of vertices at even levels, it suffices to prove \eqref{eq:hyp} for $x,y,z\in V_E$.

\medskip

Assume first that $x,y,z$ are at the same level, \emph{i.e.}, $x:=v_{2n,i}$, $y:=v_{2n,j}$ and $z:=v_{2n,k}$.
Clearly, we have $(x|z)_w - (x|y)_w = 1/2 \big(d_{G''_1}(x,y) - d_{G''_1}(x,z) \big)$ and $(z|y)_w - (x|y)_w = 1/2\big(d_{G''_1}(x,y) - d_{G''_1}(y,z)\big)$. Note that if we do not have $i\le k\le j$ or $j\le k \le i$, then Lemma \ref{l:distEx} gives $\min \big\{ (x|z)_w - (x|y)_w, (z|y)_w - (x|y)_w \big\}\le 0$. Thus, we can assume that either $i\le k\le j$ or $j\le k \le i$. Hence, we have that one of $|i-k|,|j-k|$ is greater than or equal to $|i-j|/2$ and Lemma \ref{l:distEx} gives $\min \big\{ (x|z)_w - (x|y)_w, (z|y)_w - (x|y)_w \big\}\le \a+2$.

\smallskip

Consider $(\a_1,\a_2,\a_3)$ a permutation of $(x,y,z)$.
Assume that $x,y,z$ are at two different levels.

Suppose first that there are two vertices at the closest level to $w$. Without loss of generality we can assume that $n_2:=d_{G''_1}(w,\a_1) > n_1:=d_{G''_1}(w,\a_2) = d_{G''_1}(w,\a_3)$.
Let $[\a_1\a_2],[\a_1\a_3]$ be geodesics joining $\a_1,\a_2$ and $\a_1,\a_3$, respectively, such that $\a_1'\in [\a_1\a_2]$ and $\a_1''\in[\a_1\a_3]$ are vertices at level $n_1$ with $d_{G''_1}(\a_1, \a_1') = n_2-n_1 = d_{G''_1}(\a_1, \a_1'')$.
So, by Remark \ref{r:DistLevelsDown} we have that $\a_1',\a_1''$ are either the same vertex or consecutive vertices or $d_{G''_1}(\a_1',\a_1'')= 4$, \emph{i.e.}, $d_{G''_1}(\a_1',\a_1'')\le 4$.
Hence, since $\a_1'\in[\a_1\a_2]$ and $d_{G''_1}(\a_1,\a_1')=n_2-n_1$ we have $(\a_1|\a_2)_w=1/2\big( d_{G''_1}(\a_1,\a_1') + d_{G''_1}(\a_1',w) + d_{G''_1}(w,\a_2) - d_{G''_1}(\a_1,\a_1') - d_{G''_1}(\a_1',\a_2) \big)=(\a_1'|\a_2)_w$.
Similarly, we obtain $(\a_1|\a_3)_w=(\a_1''|\a_3)_w$.
So, since $d_{G''_1}(\a_1',\a_1'')\le 4$ we have that $\left| (\a|\a_1')_w - (\a|\a_1'')_w \right|\le 4$ for every $\a\in G''_1$.
Hence, in order to obtain \eqref{eq:hyp} we can take $\{\a_1',\a_2,\a_3\}$ or $\{\a_1'',\a_2,\a_3\}$ instead of $\{\a_1,\a_2,\a_3\}$. Thus, we can assume that $x,y,z$ are at the same level, and this is a case that has already been dealt with.

Suppose now that there is exactly one vertex at the closest level of $w$.
Without loss of generality we can assume that $n_2:=d_{G''_1}(w,\a_1)=d_{G''_1}(w,\a_2) > n_1:=d_{G''_1}(w,\a_3)$.
Using the previous arguments, there are $\a_1'\in [\a_1\a_3]$ and $\a_2'\in[\a_2\a_3]$ at level $n_1$ such that $(\a_1|\a_3)_w=(\a_1'|\a_3)_w$ and $(\a_2|\a_3)_w=(\a_2'|\a_3)_w$.
So, in order to obtain \eqref{eq:hyp} by Lemma \ref{l:ProdLevels} we can take $\{\a_1',\a_2',\a_3\}$ instead of $\{\a_1,\a_2,\a_3\}$. Thus, we can  with this case assume that $x,y,z$ are at the same level, and this is a case that has already been dealt with.

\smallskip

Assume now that $\a_1,\a_2,\a_3$ are at three different levels.
Without loss of generality we can assume that $n_1 := d_{G''_1}(w,\a_1) < n_2:=d_{G''_1}(w,\a_2) < n_3:=d_{G''_1}(w,\a_3)$.
Let $[\a_3\a_1],[\a_3\a_2]$ be geodesics joining $\a_3,\a_1$ and $\a_3,\a_2$, respectively, such that there are vertices at level $n_2$,  denoted by $\a_3'\in [\a_3\a_1]$ and $\a_3''\in[\a_3\a_2]$ with $d_{G''_1}(\a_3, \a_3') = n_3-n_2 = d_{G''_1}(\a_3, \a_3'')$.
So, by Remark \ref{r:DistLevelsDown} we have that $d_{G''_1}(\a_3',\a_3'')\le 4$. Indeed, since $\a_3'\in [\a_3\a_1]$, $\a_3''\in[\a_3\a_2]$ and both are at level $n_2$ we have $(\a_3|\a_1)_w = (\a_3'|\a_1)_w$ and $(\a_3|\a_2)_w = (\a_3''|\a_2)_w$. Furthermore, it follows easily that $\left| (\a_i|\a_3')_w - (\a_i|\a_3'')_w \right|\le 2$ for $i\in\{1,2\}$ and, consequently, we can take $\{\a_1,\a_2,\a_3'\}$ or $\{\a_1,\a_2,\a_3''\}$ instead of $\{\a_1,\a_2,\a_3\}$.
Then, we can assume that $x,y,z$ are at two different levels, and this is a case that has already been dealt with.

Thus, we obtain $$\sup_{x,y,z\in V_E} \left\{ \min \big\{ (x|z)_w - (x|y)_w, (z|y)_w - (x|y)_w \big\}\right\} < \infty,$$ and consequently $G''_1$ is hyperbolic.
\end{proof}

The following results provide a counterexample for Conjecture \ref{conject}.

\begin{lemma}\label{p:HypConj}
 The graph $G'$ is hyperbolic.
\end{lemma}

\begin{proof}
Since every edge $e\in G''$ verifies $1\le L(e) \le \sqrt{5}/2$ and $G'',G''_1$ are isomorphic graphs, we have that there is a $0$-full ($\sqrt{5}/2,0$)-quasi isometry from $G''_1$ to $G''$, since
\begin{equation}\label{eq:quasisom}
  d_{G''_1}(x,y) \leq d_{G''}\big(f(x),f(y)\big)\leq \sqrt{5} \, d_{G''_1}(x,y) /2.
\end{equation}
Thus, by Theorem \ref{invarianza} and Lemma \ref{l:Ghyp} we obtain that $G''$ is hyperbolic. Now, consider $G'_1:=\overline{G'\setminus \mathbb R\times\{0,1\}}$. Clearly, $v_{0,1}$ is a cut-vertex of $G'_1$, and there is a T-decomposition of $G_1'$ with two subgraphs isomorphic to $G''$. So, by Theorem \ref{t:treedec} we obtain that $G'_1$ is hyperbolic.
Finally, it follows easily that the natural inclusion $T:G'_1\hookrightarrow G'$ is a $1$-full ($\sqrt{5}/2,0$)-quasi isometry, and Theorem \ref{invarianza} gives the result.
\end{proof}

\begin{lemma}\label{p:DistConj}
We have
\begin{equation}\label{eq:DistConj}
    \lim_{|x|\to\infty} d_{G'}\big( (x,0) , (x,1) \big)=\infty.
\end{equation}
\end{lemma}

\begin{proof}
This result is a direct consequence of Lemma \ref{l:distEx} and \eqref{eq:quasisom}, since we have
$$\lim_{n\to\infty} d_{G'}\big( v_{2n,1} \;, v_{2n,2^n} \big)=\infty.$$
\end{proof}

\begin{theorem}\label{t:conj}
 Conjecture \ref{conject} is false, \emph{i.e.}, there exists a hyperbolic tessellation graph of $\mathbb R^2$ with convex tiles.
\end{theorem}

\begin{proof}
It suffices to prove that $W$ is hyperbolic, and this is a consequence of Lemmas \ref{p:HypConj} and \ref{p:DistConj}, and Theorem \ref{t:periodic}.
\end{proof}

$W$ is a hyperbolic graph which is a tessellation graph of $\mathbb R^2$ with convex tiles (convex quadrilaterals and triangles). It follows easily that adding the largest diagonal of their quadrilaterals we obtain a hyperbolic triangulation of $\mathbb R^2$. This result is consistent with Theorem \ref{t:ConjTriang}.
However, adding to $W$ the smallest diagonal of their quadrilaterals we obtain a non-hyperbolic triangulation $T$ of $\mathbb R^2$: note that the period graph of $T$ is hyperbolic but it does not verify the condition $\lim_{|x|\to\infty} d_{T}\big( (x,0) , (x,1) \big)=\infty$, so Theorem \ref{t:periodic} gives that $T$ is not hyperbolic.

\end{document}